\documentclass{article}

\usepackage{amsmath,amssymb,amsthm}
\usepackage{a4wide}
\usepackage{hyperref}

\newtheorem{theorem}{Theorem}[section]
\newtheorem{corollary}[theorem]{Corollary}
\newtheorem{lemma}[theorem]{Lemma}
\theoremstyle{definition}
\newtheorem{definition}[theorem]{Definition}
\newtheorem{remark}[theorem]{Remark}

\begin{document}

\title{Lyapunov Functions and Stability Analysis\\ of Fractional-Order Systems\thanks{This is a 
preprint whose final form is published by Elsevier in the book 
'Mathematical Analysis of Infectious Diseases', 1st Edition -- June 1, 2022, 
ISBN: 9780323905046.}}

\author{Adnane Boukhouima$^{a}$\\
\texttt{adnaneboukhouima@gmail.com}
\and Houssine Zine$^{b}$\\
\texttt{zinehoussine@ua.pt}
\and El Mehdi Lotfi$^{a}$\\
\texttt{lotfiimehdi@gmail.com}
\and Marouane Mahrouf$^{a}$\\
\texttt{marouane.mahrouf@gmail.com}
\and Delfim F. M. Torres$^{b,}$\thanks{Corresponding author.}\\
\texttt{delfim@ua.pt}
\and Noura Yousfi$^{a}$\\
\texttt{nourayousfi@hotmail.com}}

\date{$^{a}$Laboratory of Analysis, Modeling and Simulation (LAMS),\\
Faculty of Sciences Ben M'sik, Hassan II University of Casablanca,\\
P.O. Box 7955 Sidi Othman, Casablanca, Morocco\\[0.3cm]
$^{b}$\mbox{Center for Research and Development in Mathematics and Applications (CIDMA),}\\
Department of Mathematics, University of Aveiro, 3810-193 Aveiro, Portugal}

\maketitle


\begin{abstract}
This study presents new estimates for fractional derivatives 
without singular kernels defined by some specific functions. 
Based on obtained inequalities, we give a useful method 
to establish the global stability of steady states for 
fractional-order systems and generalize some works 
existing in the literature. Finally, we apply our 
results to prove the global stability of a 
fractional-order SEIR model with a general incidence rate.\\

\bigskip

\end{abstract}


\section{Introduction}

In the last few years, the application of fractional differential equations (FDEs) 
has increased and gained much attention from researchers due to their ability 
in modeling and describing anomalous dynamics of real-world processes with memory 
and hereditary properties. Due to these properties, FDEs have been widely and 
successfully applied in various fields of science and engineering, 
such as viscoelasticity, signal and image processing, physics, mechanics, 
control, biology, and economy and finance \cite{Debnath,Sun}.

Fractional calculus (FC) literature assists to remarkable development 
of the fractional notions of differentiation. Several types of fractional derivatives 
were proposed, such as the Riemann--Liouville (RL), Caputo (C), Caputo--Fabrizio (CF) 
and Atangana--Baleanu--Caputo (ABC) operators. The standard RL
and C derivatives \cite{Kilbas} have certain disadvantages, 
being classified as fractional derivatives with singular kernels. 
Caputo and Fabrizio \cite{Caputo} suggested a new fractional derivative 
in which the memory is represented by an exponential kernel. Few years later, 
another fractional derivative was proposed by Atangana and Baleanu \cite{Atangana}, 
where the memory kernel is modeled by the Mittag--Leffler function.  
These operators are extensively used by different researchers 
to describe the dynamics of various nonlinear systems 
\cite{Bas,Boukhouima1,Hristov,Mouaouine,Sheikh,Ullah}. 

Stability is one of the powerful tools for analyzing 
the qualitative properties of non-linear dynamical systems. 
Lyapunov's direct method, also called the second Lyapunov's method, 
represents an effective way to examine the global behavior of a system 
without resolving it explicitly. This technique is based on constructing 
appropriate functionals, called the Lyapunov functionals, that should
satisfy some conditions. In physics, these functionals can be either energy, 
potential, or other, but generally there is no precise technique to determine them. 
Recently, many scholars have focused on the stability analysis of fractional-order 
systems and some others have proposed specific Lyapunov functionals candidates, 
such as Volterra-type and quadratic functions 
\cite{Aguila-Camacho,Delavari,Li,Sene,Taneco-Hernandez,De-Leon}. 
Nevertheless, these functions remain inadequate and incompatible 
with certain classes of fractional-order systems.

Motivated by the aforementioned works and observations, our main contribution 
here is to propose general Lyapunov functionals as candidates for fractional-order 
systems. We first develop new inequalities to estimate the fractional-order 
derivative of specific functions that generalize some works existing in the literature. 
These estimates allow us to construct suitable Lyapunov's functionals for 
fractional-order systems and, therefore, to establish the global stability 
of their steady-states.

The rest of the paper is structured as follows. In Section~\ref{sec:02}, 
some necessary definitions and properties related to the fractional 
calculus are recalled. Useful estimations for fractional derivatives 
are proved in Section~\ref{sec:03}. As an application of our results,
the global stability of a SEIR fractional-order model with a general 
incidence rate is studied in Section~\ref{sec:04}. Finally, 
we end with Section~\ref{sec:05} of conclusions.


\section{Preliminaries}
\label{sec:02}

In this section, we recall some definitions and properties 
of fractional operators that will be useful throughout our work. For more details, 
see \cite{Atanackovic,Atangana,Baleanu,Caputo,Kilbas,Losada,Prabhakar}.

\begin{definition}
Let $ f\in L^{1}(t_{0},+\infty) $ and $ 0<\alpha\leq 1 $. 
The Riemann--Liouville (RL) fractional integral 
of function $f$ is defined by
\begin{equation}
\label{RLI}
^{RL}I_{t_{0}}^{\alpha}f(t)
= \dfrac{1}{\Gamma(\alpha)}\int^{t}_{t_{0}}(t-x)^{\alpha-1}f(x)dx,
\end{equation}
where $ \Gamma(\cdot)$ is the Gamma function.
\end{definition}

\begin{definition}
Let $f\in H^{1}(t_{0},+\infty)$ and $0<\alpha\leq 1 $. 
The Caputo (C) fractional derivative of function $f$ is given by
\begin{equation}
\label{CD}
^{C}D_{t}^{\alpha}f(t)= \dfrac{1}{\Gamma(1-\alpha)}
\int^{t}_{t_{0}}\dfrac{f^{'}(x)}{(t-x)^{\alpha}}dx.
\end{equation}
\end{definition}

\begin{definition}
Let $f\in H^{1}(t_{0},+\infty)$ and $0<\alpha\leq 1$. 
The Caputo--Fabrizio (CF) fractional derivative 
of function $f$ is given by
\begin{equation}
\label{CFD}
^{CF}D_{t}^{\alpha}f(t)
= \dfrac{1}{2}\dfrac{B(\alpha)(2-\alpha)}{1-\alpha}
\int^{t}_{t_{0}}f^{'}(x)\exp\left[ -\dfrac{\alpha}{1-\alpha}(t-x)\right] dx,
\end{equation}
where $B(\alpha)$ denotes a normalization function obeying $B(0)=B(1)=1$.
The fractional integral associated with the CF fractional derivative is defined by
\begin{equation}
\label{CFI}
_{t}^{CF}I_{t_{0}}^{\alpha}f(t)
= \dfrac{2(1-\alpha)}{B(\alpha)(2-\alpha)}f(t)
+\dfrac{2\alpha}{B(\alpha)(2-\alpha)}\,_{t}^{RL}I_{t_{0}}^{1}f(t).
\end{equation}
\end{definition}

\begin{definition}
Let $f\in H^{1}(t_{0},+\infty)$ and $0<\alpha\leq 1$. 
The Atangana--Baleanu--Caputo (ABC) fractional derivative 
of function $f$ is given by
\begin{equation}
\label{ABCD}
_{t_{0}}^{ABC}D_{t}^{\alpha}f(t)
= \dfrac{B(\alpha)}{1-\alpha}\int^{t}_{t_{0}}f^{'}(x)E_\alpha\left[ 
-\dfrac{\alpha}{1-\alpha}(t-x)^{\alpha}\right]dx.
\end{equation} 
The fractional integral associated with 
the ABC fractional derivative is defined by
\begin{equation}
\label{ABI}
_{t}^{AB}I_{t_{0}}^{\alpha}f(t)
= \dfrac{1-\alpha}{B(\alpha)}f(t)
+\dfrac{\alpha}{B(\alpha)}\,_{t}^{RL}I_{t_{0}}^{\alpha}f(t).
\end{equation}
\end{definition}

\begin{definition}
Let $\alpha > 0$ and $\beta>0$. The Mittag-Leffler 
function of two parameters $\alpha$ and $\beta$ 
is defined by 
$$
E_{\alpha,\beta}(z)=\sum_{j=0}^{\infty}
\dfrac{z^{j}}{\Gamma(\alpha j+\beta)}, \; z \in \mathbb{C}. 
$$
\end{definition}

\begin{remark}
If $\beta=1$, then we have
$$
E_{\alpha,1}(z) = E_{\alpha}(z)=\sum_{j=0}^{\infty}\dfrac{z^{j}}{\Gamma(\alpha j+1)},
$$
which is called the Mittag-Leffler function of one parameter $\alpha$;
if $\alpha=\beta =1$, then one gets
$$
E_{1,1}(z)=\sum_{j=0}^{\infty}\dfrac{z^{j}}{j!}=\exp(z).
$$
\end{remark}

\begin{theorem}
The derivative of the Mittag-Leffler function satisfies:
$$ 
\dfrac{dE_{\alpha,\beta}}{dz}(z)= E^{2}_{\alpha,\alpha+\beta}(z).
$$
\end{theorem}

In \cite{Aguila-Camacho,Sene,Taneco-Hernandez}, the authors prove 
the following inequalities for  estimating the fractional derivative 
of certain functions.

\begin{lemma}
\label{Lemma3.1}
Let $u(t)$ be a real continuous and differentiable function. 
Then, for any $t \geq t_{0}$ and $0<\alpha\leq1$, we have
\begin{equation}
\label{18}
_{t_{0}}^{ABC}D_{t}^{\alpha}\left( u^{2}(t)\right) 
\leq 2u(t)\, ^{ABC}_{t_{0}}D_{t}^{\alpha}u(t),
\end{equation}
\begin{equation}
\label{19}
_{t_{0}}^{CF}D_{t}^{\alpha}\left(u^{2}(t)\right) 
\leq 2u(t)\, ^{CF}_{t_{0}}D_{t}^{\alpha}u(t),
\end{equation}
\begin{equation}
\label{20}
_{t_{0}}^{C}D_{t}^{\alpha}\left(u^{2}(t)\right) 
\leq 2u(t)\, ^{C}_{t_{0}}D_{t}^{\alpha}u(t).
\end{equation}
\end{lemma}

\begin{lemma}
\label{Lemma3.2}
Let $u(t)$ be a positive real continuous and differentiable function. 
Then, for any $t \geq t_{0}$, $0<\alpha\leq1$, and $u^{*}>0$,
one has
\begin{align}
\label{10}
_{t_{0}}^{ABC}D_{t}^{\alpha}\left[ u(t)-u^*
-u^*\ln\dfrac{u(t)}{u^*}\right]& 
\leq \left( 1-\dfrac{u^*}{u(t)}\right)
\: _{t_{0}}^{ABC}D_{t}^{\alpha}u(t),\\
\label{11}
_{t_{0}}^{C}D_{t}^{\alpha}\left[ u(t)-u^*
-u^*\ln\dfrac{u(t)}{u^*}\right]& 
\leq \left( 1-\dfrac{u^*}{u(t)}\right)
\: _{t_{0}}^{C}D_{t}^{\alpha}u(t).
\end{align}
\end{lemma}


\section{Useful fractional derivative estimates}
\label{sec:03}

The aim of this section is to establish some new estimates 
for the fractional derivative of function $\Psi$ defined by
\begin{equation}
\label{12}
\begin{split}
\Psi(u)&=\int^{u}_{u^{*}}\dfrac{g(s)-g(u^{*})}{g(s)}ds\\
&=u-u^{*}-\int^{u}_{u^{*}}
\dfrac{g(u^{*})}{g(s)}ds,
\end{split}
\end{equation}
where $g$ is a non-negative, differentiable, and strictly 
increasing function on $\mathbb{R}^{+}$. Our estimates will 
allow us to extend the classical Lyapunov functions 
to fractional-order systems.

Note that $\Psi$ is positive in $\mathbb{R}^{+}
\setminus \lbrace u^{*}\rbrace$ with $\Psi(u^{*})=0$. 
In fact, $\Psi$ is differentiable and 
$$ 
\dfrac{d\Psi }{du}=1-\dfrac{g(u^{*})}{g(u)}. 
$$
Since $g$ is a strictly increasing function, then $\Psi$ 
is strictly decreasing if $u< u^{*}$ and strictly increasing 
if $u> u^{*}$, with $u^{*}$ its global minimum. 

\begin{theorem}
\label{proposition1}
Let $u(t)$ be a real positive differentiable function. 
Then, for any $t \geq t_{0}$, $0<\alpha\leq1$, and $u^{*}>0$, we have
\begin{equation}
\label{13}
_{t_{0}}^{ABC}D_{t}^{\alpha}\Psi(u(t))
\leq \left( 1-\dfrac{g(u^{*})}{g(u(t))}\right)
\, ^{ABC}_{t_{0}}D_{t}^{\alpha}u(t),
\end{equation}
\begin{equation}
\label{17}
_{t_{0}}^{CF}D_{t}^{\alpha}\Psi(u(t))
\leq \left( 1-\dfrac{g(u^{*})}{g(u(t))}\right)
\, ^{CF}_{t_{0}}D_{t}^{\alpha}u(t).
\end{equation}
\end{theorem}

\begin{proof}
We start by reformulating inequality \eqref{13}. 
By the linearity of the ABC fractional derivative, 
we obtain that
\begin{equation*}
_{t_{0}}^{ABC}D_{t}^{\alpha}\Psi(u(t))
=\:_{t_{0}}^{ABC}D_{t}^{\alpha}u(t)-\:_{t_{0}}^{ABC}D_{t}^{\alpha}\left[ 
\int_{u^*}^{u(t)}\dfrac{g(u^*)}{g(s)}ds\right].
\end{equation*}
Hence, the inequality \eqref{13} becomes
\begin{equation*}
_{t_{0}}^{ABC}D_{t}^{\alpha}u(t)-\:_{t_{0}}^{ABC}D_{t}^{\alpha}\left[ 
\int_{u^*}^{u(t)}\dfrac{g(u^*)}{g(s)}ds\right]
\leq \left( 1-\dfrac{g(u^*)}{g(u)}\right)
\: _{t_{0}}^{ABC}D_{t}^{\alpha}u(t).
\end{equation*}
Because $g$ is a non-negative function, we get
\begin{equation*}
g(u(t))\left[ _{t_{0}}^{ABC}D_{t}^{\alpha}u(t)
- \,_{t_{0}}^{ABC}D_{t}^{\alpha}\left( 
\int_{u^*}^{u(t)}\dfrac{g(u^*)}{g(s)}ds\right)\right] 
\leq \left( g(u(t))-g(u^*)\right)\, _{t_{0}}^{ABC}D_{t}^{\alpha}u(t).
\end{equation*}
Thus,  
\begin{equation}
\label{14}
_{t_{0}}^{ABC}D_{t}^{\alpha}u(t)-g(u(t))_{t_{0}}^{ABC}D_{t}^{\alpha}\left[ 
\int_{u^*}^{u(t)}\dfrac{1}{g(s)}ds\right]
\leq 0.
\end{equation}
Using the definition of ABC fractional derivative \eqref{ABCD}, we have
\begin{equation*}
_{t_{0}}^{ABC}D_{t}^{\alpha}u(t)
=\dfrac{B(\alpha)}{1-\alpha}
\int^{t}_{t_{0}}u'(x)E_\alpha\left[ 
-\dfrac{\alpha}{1-\alpha}(t-x)^{\alpha}\right]  dx
\end{equation*}
and
\begin{equation*}
_{t_{0}}^{ABC}D_{t}^{\alpha}\left[ 
\int_{u^*}^{u(t)}\dfrac{1}{g(s)}ds\right]
=\dfrac{B(\alpha)}{1-\alpha}\int^{t}_{t_{0}}
\dfrac{u'(x)}{g(u(x))}E_\alpha\left[ 
-\dfrac{\alpha}{1-\alpha}(t-x)^{\alpha}\right] dx.
\end{equation*}
Consequently, the inequality \eqref{14} can be written as
\begin{equation}
\label{15}
\dfrac{B(\alpha)}{1-\alpha}
\int^{t}_{t_{0}}u'(x)\left(1-\dfrac{g(u(t))}{g(u(x))}\right) 
E_\alpha\left[ -\dfrac{\alpha}{1-\alpha}(t-x)^{\alpha}\right] dx
\leq 0.
\end{equation}
Now, we show that the inequality \eqref{15}
is verified. For this, we denote
$$
H(t)= \int^{t}_{t_{0}}u'(x)\left(1-\dfrac{g(u(t))}{g(u(x))}\right) 
E_\alpha\left[ -\dfrac{\alpha}{1-\alpha}(t-x)^{\alpha}\right] dx
$$ 
and set
\begin{equation*}
v(x,t)= E_\alpha\left[ -\dfrac{\alpha}{1-\alpha}(t-x)^{\alpha}\right]; 
\quad \dfrac{dv(x,t)}{dx}=\dfrac{\alpha^2(t-x)^{\alpha-1}}{1-\alpha} 
E_{\alpha,\alpha+1}^{2} \left[ -\dfrac{\alpha}{1-\alpha}(t-x)^{\alpha}\right];
\end{equation*}
\begin{equation*}
w(x,t)=u(x)-u(t)-\int^{u(x)}_{u(t)}\dfrac{g(u(t))}{g(s)}ds; 
\quad\dfrac{dw(x,t)}{dx}=u'(x)\left( 1-\dfrac{g(u(t))}{g(u(x))}\right).
\end{equation*}
Integrating by parts the integral $H(t)$, we obtain that
\begin{equation}
\label{16}
\begin{split}
H(t) &= \left[ E_\alpha\left[ 
-\dfrac{\alpha}{1-\alpha}(t-x)^{\alpha}\right]w(x,t) \right]^{x=t}_{x=t_0}\\
&\quad -\int^{t}_{t_{0}}\dfrac{\alpha^2(t-x)^{\alpha-1}}{1-\alpha} 
E_{\alpha,\alpha+1}^{2} \left[ -\dfrac{\alpha}{1-\alpha}(t-x)^{\alpha}\right]w(x,t) dx.
\end{split}
\end{equation}
Since $w(x,t)\geq 0$ and
\begin{equation*}
\lim_{x\to t}E_\alpha\left[ 
-\dfrac{\alpha}{1-\alpha}(t-x)^{\alpha}\right]w(x,t)=0,
\end{equation*}
it follows that
\begin{equation*}
\begin{split}
H(t)
&=- E_\alpha\left[ -\dfrac{\alpha}{1-\alpha}(t-t_0)^{\alpha}\right]w(t_{0},t)\\
&\quad -\int^{t}_{t_{0}}\dfrac{\alpha^2(t-x)^{\alpha-1}}{1-\alpha} 
E_{\alpha,\alpha+1}^{2} \left[ 
-\dfrac{\alpha}{1-\alpha}(t-x)^{\alpha}\right]w(x,t) dx\leq 0.
\end{split}
\end{equation*}
As a result, the inequality \eqref{15} is satisfied and \eqref{13} holds true.
\end{proof}

\begin{remark}
Inequality \eqref{17} is obtained by replacing 
$E_\alpha\left[ -\dfrac{\alpha}{1-\alpha}(t-x)^{\alpha}\right]$ 
with $\exp\left[ -\dfrac{\alpha}{1-\alpha}(t-x)\right]$ 
and following the same steps as given in the proof
of Theorem~\ref{proposition1}.
\end{remark}

\begin{remark}
\label{Remark3}
A similar inequality also holds for the Caputo fractional 
derivative as follows \cite{Boukhouima}:
\begin{align}
\label{VTE1}
_{t_{0}}^{C}D_{t}^{\alpha}\Psi(u(t))&
\leq \left( 1-\dfrac{g(u^{*})}{g(u(t))}\right)
\, ^{C}_{t_{0}}D_{t}^{\alpha}u(t).
\end{align}
\end{remark}

If $g(s) =s$, then we obtain $ \Psi(u(t))= u(t)-u^*-u^*\ln\dfrac{u(t)}{u^*}$. 
We obtain from Theorem~\ref{proposition1} the following corollary.

\begin{corollary}
Let $u(t)$ be a positive differentiable function. 
For any $t \geq t_{0}$, $0<\alpha\leq1$, and $u^{*}>0$, we have
\begin{equation}
\label{VTE}
\begin{split}
_{t_{0}}^{ABC}D_{t}^{\alpha}\left[ u(t)-u^*
-u^*\ln\dfrac{u(t)}{u^*}\right]& 
\leq \left( 1-\dfrac{u^*}{u(t)}\right)
\: _{t_{0}}^{ABC}D_{t}^{\alpha}u(t),\\
_{t_{0}}^{CF}D_{t}^{\alpha}\left[ u(t)-u^*
-u^*\ln\dfrac{u(t)}{u^*}\right]& 
\leq \left( 1-\dfrac{u^*}{u(t)}\right)
\: _{t_{0}}^{CF}D_{t}^{\alpha}u(t).
\end{split}
\end{equation}
\end{corollary}


\section{An application}
\label{sec:04}

In \cite{Yang}, Yang and Xu proposed a SEIR model with Caputo 
fractional derivative and general incidence rate as follows:
\begin{equation}
\label{SEIR}
\begin{cases}
_{0}^{C}D_{t}^{\alpha}S(t) 
= \Lambda^{\alpha}  - d^{\alpha} S(t)- \beta^{\alpha} 
F\left(S(t)\right) G\left( I(t)\right),\\[0.2 cm]
_{0}^{C}D_{t}^{\alpha}E(t) = \beta^{\alpha} 
F\left(S(t)\right) G\left( I(t)\right) 
- (\sigma^{\alpha} + d^{\alpha})E(t), \\[0.2 cm]
_{0}^{C}D_{t}^{\alpha}I(t) = \sigma^{\alpha} E(t) 
- (\gamma^{\alpha} + d^{\alpha})I(t),\\[0.2 cm]
_{0}^{C}D_{t}^{\alpha}R(t) 
=  \gamma^{\alpha} I(t) - d^{\alpha} R(t),
\end{cases}
\end{equation}
where $\alpha\in(0,1]$. The variables $S(t)$, $E(t)$, $I(t)$ and $R(t)$ 
represent the number of susceptible, exposed, infective and recovered 
individuals at time $t$, respectively. All the other parameters 
are assumed to be positive constants.

The authors of \cite{Yang} first analyzed the global stability 
of the disease-free equilibrium and discussed the stability 
of the endemic equilibrium but only when $F(S)=S$. However, 
they mentioned that they can not use the estimation \eqref{11} 
in Lemma~\ref{Lemma3.2} to establish global stability in the general 
case and kept this problem as an open question, for future work. 

In addition, we note that function $F\left(S(t)\right) G\left( I(t)\right)$ 
does not cover all the incidence functions existing in the literature, 
e.g., $\dfrac{SI}{1+\alpha_1S+\alpha_2 I+\alpha_3 SI}$, 
$\alpha_1, \alpha_2,\alpha_3\geq 0$ \cite{Lotfi,Mahrouf}, 
where we can not separate the variables $S$ and $I$. Here, we generalize 
the SEIR model \eqref{SEIR} and apply our results to give a rigorous 
proof of the stability for both equilibrium points. 

Let us consider the general fractional-order SEIR model 
\begin{equation}
\label{FSEIR}
\begin{cases}
_{0}D_{t}^{\alpha}S(t) = \Lambda^{\alpha}  - d^{\alpha} S(t)
-  F\left(S(t),I(t)\right),\\[0.2 cm]
_{0}D_{t}^{\alpha}E(t) = F\left(S(t),I(t)\right) 
- (\sigma^{\alpha} + d^{\alpha})E(t), \\[0.2 cm]
_{0}D_{t}^{\alpha}I(t) = \sigma^{\alpha} E(t) 
- (\gamma^{\alpha} + d^{\alpha})I(t),\\[0.2 cm]
_{0}D_{t}^{\alpha}R(t) 
=  \gamma^{\alpha} I(t) - d^{\alpha} R(t),
\end{cases}
\end{equation}
where $_{0}D_{t}^{\alpha}$ denotes any fractional-order 
derivative mentioned in Section~\ref{sec:02}. The general incidence function 
$F:\mathbb{R}^{2}_{+}\to\mathbb{R}_{+}$ is assumed to be continuously 
differentiable and to satisfy the following hypotheses:
\begin{equation}
\label{H}
\tag{$H$} 
\begin{gathered}
F(S,0)= F(0,I)=0 \: \text{ and }\: F(S,I) = I F_{1}(S,I)
\quad \text{ for all } S,I \geq 0,\\
\frac{ \partial F_{1}}{\partial S}(S,I)> 0 \: \text{ and }\: 
\frac{ \partial F_{1}}{\partial I}(S,I)\leq 0 
\quad \text{ for all } S\geq 0\ \text{ and } \ I \geq 0,\\
\frac{ \partial F}{\partial I}(S,I)\geq 0 
\quad \text{ for all } S\geq 0\ \text{ and } \ I \geq 0.
\end{gathered}
\end{equation}
Since $R(t)$ does not appear in the first three equations 
of system \eqref{FSEIR}, without loss of generality
we discuss the following system:
\begin{equation}
\label{FSEIRM}
\begin{cases}
_{0}D_{t}^{\alpha}S(t) 
= \Lambda^{\alpha}  - d^{\alpha} S(t)-  F\left(S(t),I(t)\right),\\[0.2 cm]
_{0}D_{t}^{\alpha}E(t) 
= F\left(S(t),I(t)\right) - m_{1}E(t), \\[0.2 cm]
_{0}D_{t}^{\alpha}I(t) 
= \sigma^{\alpha} E(t) - m_{2}I(t),
\end{cases}
\end{equation}
where $m_{1} = \sigma^{\alpha} + d^{\alpha}$ 
and $m_{2} = \gamma^{\alpha} + d^{\alpha}$.

System \eqref{FSEIRM} has a disease-free equilibrium $P_{f} = (S_{0}, 0, 0)$ 
with $S_{0} =\dfrac{\Lambda^{\alpha}}{d^{\alpha}}$  
and an endemic equilibrium $P^{*} = (S^{*},E^{*}, I^{*})$ 
when $R_0 > 1$, where 
$$ 
R_0= \dfrac{\sigma^{\alpha}}{m_{1} m_{2} }\frac{ \partial F(S_{0},0)}{\partial I}
$$
and
$E^{*}\in\left[ 0, \dfrac{\Lambda^{\alpha}}{d^{\alpha}} \right]$, 
$S^{*}=\dfrac{\Lambda^{\alpha}-m_{1} E^{*}}{d^{\alpha}}$ 
and $I^{*}= \dfrac{\sigma^{\alpha}E^{*}}{m_{2}}$.

Next, we prove the global stability of both equilibriums 
by constructing appropriate Lyapunov functionals 
and using our results of Section~\ref{sec:03}.

\begin{theorem}
The disease-free equilibrium $P_f$ is asymptotically stable when $R_0\leq 1$.
The endemic equilibrium $ P^{*} $ is asymptotically stable whenever $R_0>1$.
\end{theorem}

\begin{proof}
For the disease-free equilibrium 
we define the following Lyapunov functional:
\begin{align*}
V_{0}(t)=\int^{S(t)}_{S_0}\dfrac{F_{1}(x,0)-F_{1}(S_0,0)}{F_{1}(x,0)}dx
+ E(t) + \dfrac{m_{1}}{\sigma^{\alpha}}I.
\end{align*} 
Applying our results, we estimate the fractional time 
derivative of function $V_{0}$ as 
\begin{align*}
_{0}D_{t}^{\alpha}V_{0}(t)\leq \left( 1 - \dfrac{F_{1}(S_0,0)}{F_{1}(S(t,0)}\right)
\, _{0}D_{t}^{\alpha}S(t)+ \,_{0}D_{t}^{\alpha}E(t) 
+ \dfrac{m_{1}}{\sigma^{\alpha}}\,_{0}D_{t}^{\alpha}I(t).
\end{align*}
Using the fact that $\Lambda^{\alpha}= d^{\alpha}S_0$, we get
\begin{align*}
_{0}D_{t}^{\alpha}V_{0}(t)&\leq \left( 1 -
\dfrac{F_{1}(S_0,0)}{F_{1}(S(t),0)}\right)(d^{\alpha}(S_0-S(t))
+I(t)F_{1}(S_0,0)\dfrac{F_{1}(S(t),I(t))}{F_{1}(S(t),0)} 
- \dfrac{m_{1}m_{2}}{\sigma^{\alpha}}I(t)\\
&\leq \left( 1 - \dfrac{F_{1}(S_0,0)}{F_{1}(S(t),0)}\right)
(d^{\alpha}(S_0-S(t))+I(t)F_{1}(S_0,0) 
- \dfrac{m_{1}m_{2}}{\sigma^{\alpha}}I(t)\\
&= \left( 1 - \dfrac{F_{1}(S_0,0)}{F_{1}(S(t),0)}\right)
(d^{\alpha}(S_0-S(t))+\left( \dfrac{\partial F
(S_0,0)}{\partial I} - \dfrac{m_{1}m_{2}}{\sigma^{\alpha}}\right) I(t)\\
&= \left( 1 - \dfrac{F_{1}(S_0,0)}{F_{1}(S(t),0)}\right)(d^{\alpha}(S_0-S(t))
+\dfrac{m_{1}m_{2}}{\sigma^{\alpha}}\left( R_0 - 1\right) I(t).
\end{align*}
Since $F_{1}$ is an increasing function with respect to $S$, one has
\begin{align*}
1-\dfrac{F_{1}(S_{0},0)}{F_{1}(S,0)}
&\geq 0 \quad \text{for}\quad S\geq S_0,\\
1-\dfrac{F_{1}(S_{0},0)}{f(S,0)}
&< 0 \quad \text{for}\quad S<S_0.
\end{align*}
Then, we get 
\begin{equation*}
\left(1-\dfrac{F_{1}(S_{0},0)}{F_{1}(S,0)}\right)(S_0-S)\leq 0.
\end{equation*}
It follows that $_{0}D_{t}^{\alpha}V_{0}(t)\leq 0$ for 
$R_0\leq 1 $ with $_{0}D_{t}^{\alpha}V_{0}(t)= 0$ 
if $S=S_0$ and $I=0$.  Substituting $(S,I)=(S_0,0)$ in \eqref{FSEIRM} 
shows that $ E\to 0 $ as $ t\to \infty $. We conclude that the disease-free 
equilibrium $ P_f $ is  asymptotically stable when $ R_0\leq 1 $.

Next, we assume that $R_0>1$ and we propose the following Lyapunov 
functional $V_1$ for the endemic equilibrium:
\begin{align*}
V_{1}(t)=\int^{S(t)}_{S^{*}}\dfrac{F(x,I^{*})-F(S^{*},I^{*})}{F(x,I^{*})}dx
+\int^{E(t)}_{E^{*}}\dfrac{x-E^{*}}{x}dx + \dfrac{m_{1}}{\sigma^{\alpha}}\left( 
\int^{I(t)}_{I^{*}}\dfrac{x-I^{*}}{x}dx \right).
\end{align*}
Computing the time fractional derivative of $V_{1}$, we get 
\begin{align*}
_{0}D_{t}^{\alpha}V_{1}(t)\leq &
\left( 1 - \dfrac{F(S^{*},I^{*})}{F(S(t),I^{*})}\right)
\, _{0}D_{t}^{\alpha}S(t)+\left( 1 - \dfrac{E^{*}}{E}\right) 
\,_{0}D_{t}^{\alpha}E(t)\\
&+ \dfrac{m_{1}}{\sigma^{\alpha}}\left( 
1 - \dfrac{I^{*}}{I}\right)\,_{0}D_{t}^{\alpha}I(t).
\end{align*}
Using the fact that $\Lambda^{\alpha}=d^{\alpha}S^{*}+ F(S^{*},I^{*})$, 
$ F(S^{*},I^{*})=m_{1}E^{*}$ and $\sigma^{\alpha}E^{*}= m_{2}I^{*}$, we obtain
\begin{align*}
_{0}D_{t}^{\alpha}V_{1}(t)\leq 
&\left( 1 - \dfrac{F(S^{*},I^{*})}{F(S(t),I^{*})}\right)
d^{\alpha}(S^{*}-S(t))\\
&+F(S^{*},I^{*})\left[ 3- \dfrac{F(S^{*},I^{*})}{F(S,I^{*})}
+ \dfrac{F(S,I)}{F(S,I^{*})}-\dfrac{E^{*}F(S,I)}{EF(S^{*},I^{*})}
-\dfrac{I}{I^{*}}-\dfrac{I^{*}E}{IE^{*}}\right] \\
= &\left( 1 - \dfrac{F_{1}(S^{*},I^{*})}{F_{1}(S(t),I^{*})}\right) 
d^{\alpha}(S^{*}-S(t))-F(S^{*},I^{*})\left[G\left( \dfrac{I}{I^{*}}\right)
-G\left( \dfrac{F(S,I)}{F(S,I^{*})}\right)\right. \\
&\left.   +G\left( \dfrac{F(S^{*},I^{*})}{F(S,I^{*})}\right) 
+ G\left( \dfrac{E^{*}F(S,I)}{EF(S^{*},I^{*})}\right)
+G\left( \dfrac{I^{*}E}{IE^{*}}\right)  \right], 
\end{align*}
where $G(x)=x-1-\ln(x)$. Now, we show that 
$G\left( \dfrac{I}{I^{*}}\right)-G\left( \dfrac{F(S,I)}{F(S,I^{*})}\right)\geq 0$. 
For this, we set 
$$
H(I)=  G\left( \dfrac{F(S,I)}{F(S,I^{*})}\right)-G\left( \dfrac{I}{I^{*}}\right).
$$ 
Computing the derivative of $H$ with respect to $I$, we obtain 
$$
\dfrac{dH}{dI}= \dfrac{F(S,I)-F(S,I^{*})}{F(S,I)F(S,I^{*})}
\dfrac{\partial F (S,I)}{\partial I}-\dfrac{I-I^{*}}{II^{*}}. 
$$ 
We discuss two cases:
\begin{description}
\item[Case 1.] If $ I\geq I^{*}$, then 
$F(S,I)\geq F(S,I^{*})$. Because
$$ 
\dfrac{\partial F (S,I)}{\partial I}
= F_{1} (S,I)+ I \dfrac{\partial F_{1} (S,I)}{\partial I} \leq F_{1} (S,I)
$$
it follows that
\begin{align*}
\dfrac{dH}{dI}
& \leq  \dfrac{F(S,I)-F(S,I^{*})}{F(S,I)F(S,I^{*})}F_{1} (S,I)
-\dfrac{I-I^{*}}{II^{*}}\\
& = \dfrac{1}{F(S,I^{*})}\left[ F_{1}(S,I)-F_{1}(S,I^{*})\right]\leq 0. 
\end{align*}
Hence $ H(I)\leq H(I^{*})=0$.

\item[Case 2.]
If $ I\leq I^{*}$, then $F(S,I)\leq F(S,I^{*})$. Therefore,
\begin{align*}
\dfrac{dH}{dI}
& \geq   \dfrac{F(S,I)-F(S,I^{*})}{F(S,I)F(S,I^{*})}F_{1} (S,I)-\dfrac{I-I^{*}}{II^{*}}\\
& = \dfrac{1}{F(S,I^{*})}\left[ F_{1}(S,I)-F_{1}(S,I^{*})\right]\geq 0. 
\end{align*}
Hence, $ H(I)\leq H(I^{*})=0$.
\end{description}
We conclude that $_{0}D_{t}^{\alpha}V_{1}(t)$ is negative definite. 
Consequently, the endemic equilibrium $ P^{*} $ is asymptotically stable 
whenever $R_0>1$.
\end{proof}


\section{Conclusions}
\label{sec:05}

In this paper, we have developed some estimates for fractional derivatives 
without a singular kernel and applied it to establish 
the stability of fractional-order systems. To illustrate the efficacy 
of the obtained results, we have employed them to solve an open problem
posed by Yang and Xu in \cite{Yang} and prove the stability of a SEIR 
fractional-order system with a general incidence rate. We 
construct suitable Lyapunov functionals and proved the 
globally asymptotically stability of the disease-free and endemic equilibriums 
in terms of the basic reproduction number $R_0$. Our results 
generalize and improve those of \cite{Almeida,Gonzalez-Parra,Sene1}. 


\section*{Acknowledgments}

H.Z. and D.F.M.T. were supported by FCT within project UIDB/04106/2020 (CIDMA).



\bigskip


\end{document}